\documentclass[10pt]{amsart}
\usepackage{amsfonts}
\usepackage{amssymb}
\usepackage{amsmath}
\usepackage[arrow,matrix,curve]{xy}

\newcommand{\C}{{\mathbb C}}
\newcommand{\R}{{\mathbb R}}
\newcommand{\Z}{{\mathbb Z}}
\newcommand{\N}{{\mathbb N}}

\newcommand{\F}{{\mathbb F}}

\newtheorem{theorem}{Theorem}[section]
\newtheorem{lemma}[theorem]{Lemma}
\newtheorem{remark}[theorem]{Remark}
\newtheorem{example}[theorem]{Example}
\newtheorem{corollary}[theorem]{Corollary}
\newtheorem{proposition}[theorem]{Proposition}
\newtheorem{definition}[theorem]{Definition}

\newtheorem{examples}[theorem]{Examples}

\def\cal{\mathcal}

\newcommand{\calr}[0]{{\cal R}}

\newcommand{\calg}[0]{{\cal G}}

\newcommand{\cale}[0]{{\cal E}}
\newcommand{\calk}[0]{{\cal K}}
\newcommand{\calo}[0]{{\cal O}}

\newcommand{\cals}[0]{{\cal S}}
\newcommand{\calm}[0]{{\cal M}}

\newcommand{\calu}[0]{{\cal U}}

\newcommand{\calt}[0]{{\cal T}}

\begin{document}

\title{Some remarks in $C^*$- and $K$-theory}
\author[B. Burgstaller]{Bernhard Burgstaller}
\email{bernhardburgstaller@yahoo.de}

\begin{abstract}
This note consists of three unrelated remarks.
First, we demonstrate
how roughly speaking $*$-homomorphisms between
matrix stable $C^*$-algebras are exactly the
uniformly continuous $*$-preserving group homomorphisms between their
genral linear groups.
Second, using the Cuntz picture in $KK$-theory we bring
morphisms in $KK$-theory represented by generators and relations
to a particular simple form.
%
%
%
Third, we show that for an inverse semigroup
%
its associated groupoid is Hausdorff if and only if the
inverse semigroup is $E$-continuous.
\end{abstract}

\maketitle



\section{Introduction}

In this note we present three unrelated results in $C^*$-theory and $K$-theory.
%
The first result is demonstrated in Section \ref{section2} and shows that for all unital $C^*$-algebras $A$ and $B$, every uniformly continuous, $*$-preserving group homomorphism
$\varphi:GL(A \otimes M_2) \rightarrow GL(B)$ can be extended to a $*$-homomorphism $A \otimes M_2 \rightarrow B$,
provided a very light additional technical condition for the restriction
of $\varphi$ to the complex numbers 
is satisfied, see Corollary \ref{corollary22} and Section \ref{section2}.
Actually, we have demonstrated a similar result already in \cite{burgmatrix}, but the improvement,
thanks to some trick by L. Moln\'ar \cite{molnar},
is that the additional technical condition is here subjectively somewhat easier,
even if not strictly logically comparable with the one in \cite{burgmatrix}.

In the next Section \ref{section3}, we make a turn to $KK$-theory \cite{kasparov1988}. J. Cuntz
\cite{cuntz1984} and N. Higson \cite{higson} found out that Kasparov's $KK$-theory
is the universal stable, homotopy invariant, split-exact functor from the $C^*$-category
to an additive category.
This makes it possible to describe $KK$-theory as a localization of the category of $C^*$-algebras, or expressed in less technical terms, by adding certain synthetical inverses
to the category of $C^*$-algebras and moding out certain relations to form $KK$-theory.
We slightly simplify the representation of $KK$-elements in this picture at first, but make the most dramatical simplification
by using the Cuntz-picture \cite{cuntz1984,cuntzpicture} of $KK$-theory elements.
This picture of $KK$-theory may also be analogously and readily defined equivariantly for other equivariant structures
than groups, say semigroups, categories and so on, and even the category of $C^*$-algebras may be changed to other
(topological) algebras.

In the last Section \ref{section4} 
we observe
that a discrete inverse semigroup induces a {Hausdorff} groupoid if and only if the inverse semigroup
is {$E$-continuous}.
We also note that both equivalent technical conditions appear necessary to define a non-degenerate, $C_0(X)$-compatible $C_0(X)$-valued $L^2(G)$-module, see Definition \ref{definitionEcontinuous} and Example \ref{example1} for more on this.
Such a module is a 
useful tool for the computation of the $K$-theory of inverse semigroup crossed products.
However, the lack of such a module in the non-Hausdorff case hinders the computation
of beformentioned $K$-theory groups of crossed products by non-applicability of parallel
methods successful in the group case. The difficulty of computation
has been already observed by Tu \cite{tunonhausdorffgroupoid} for the more
general setting of non-Hausdorff groupoids in the context of Baum--Connes theory.

All chapters in this note can be read
completely independently.

\section{Group and algebra homomoprhisms}	\label{section2}

In this section we show how certain group homomorphisms between the group of invertible elements of $C^*$-algebras
can be extended to $*$-homomorphisms.
A map $\varphi:A \rightarrow B$ between $C^*$-algebras $A$ and $B$ is called a {\em $*$-semigroup homomorphism} if it is multiplicative (i.e. $\varphi(ab)=\varphi(a)\varphi(b)$) and {\em $*$-preserving}
(i.e. $\varphi(a^*)=\varphi(a)^*$). As usual, $M_n$ denotes the $C^*$-algebra of all complex-valued $n \times n$-matrices, and $GL(A)$ the general linear group of $A$.

\begin{proposition}	\label{propm2}
Let $\varphi:GL(A \otimes M_2) \rightarrow B$ be an arbitrary function where
$A$ and $B$ are $C^*$-algebras and $A$ is unital.
Then the following are equivalent:

\begin{itemize}
\item[(a)]
$\varphi$ extends to a $*$-homomorphism $A \otimes M_2 \rightarrow B$.

\item[(b)]
$\varphi$ is a uniformly continuous, $*$-semigroup homomorphism with
\begin{equation}   \label{eq22}
\|\varphi(1/2)\| < 1, \quad   \varphi(i1)=i 1.
\end{equation}

\end{itemize}

\end{proposition}

\begin{remark} \label{remark1}
{\rm
Alternatively, instead of requiring $\|\varphi(1/2)\| < 1$ in Proposition \ref{propm2}.(b), we may equivalently require that $\|\varphi(z)\| < 1$ for any single fixed 
$z \in GL(A \otimes M_n)$ with $\|z\| < 1$.
}
\end{remark}

\begin{proof}
%
%

(a) to (b) is clear. To show (b) to (a), we are going to apply \cite[Proposition 2.6]{}.
At first we continuously extend $\varphi$ to an equally denoted function $\varphi:\overline{GL(A\otimes M_2)} \rightarrow B$ (norm closure) by using Cauchy sequences
and the uniform continuity of $\varphi$. Then $\varphi$ is a $*$-semigroup homomorphism.
Notice that $\varphi(0) =  \lim_{n} \varphi(z^n)=0$ 
by 
Remark \ref{remark1}.
By applying Proposition 2.6 of \cite{burgmatrix} we are done when showing the ortho-additivity relation
$\varphi(e_{11}+e_{22}) = \varphi(e_{11}) + \varphi(e_{22})$, where $e_{ii}$ are the standard matrix corners.
To this end,
we use the following
trick by L. Moln\'ar \cite{molnar} by means of the exponential function,
which we are going to recall for convenience of the reader.

Consider the $C^*$-subalgebra $B'$ of $B$ generated by the image of $\varphi$. It is unital with unit $\varphi(1)$.
Represent $B'$ faithfully on a Hilbert space $H$ such that $1_{B(H)}$ is the unit of $B'$.
In the following, identify now $B'$ as a subalgebra of $B(H)$. 


Let $P$ be a projection in $M_2(A)$. Clearly $e^{\lambda P}$ is invertible for every $\lambda \in \R$ and so in the domain of $\varphi$.
Consider the map $\lambda \mapsto \varphi(e^{\lambda P}) = \varphi(1-P + e^{\lambda} P)$
from $\R$ into $GL(B(H))$.
This is a one-parameter group.
Thus there exists an operator $T \in B(H)$ such that
$$\varphi(1-P + e^{\lambda} P) = e^{\lambda T}.$$
Since $\varphi$ is $*$-preserving, $e^{\lambda T}$ is self-adjoint for all $\lambda \in \R$.
This implies that $T$ is also self-adjoint.
By the uniform continuity of $\varphi$, for every $\varepsilon >0$
there exists a $\delta >0$ such that
$$\|e^{\lambda T}- e^{\mu T}\| = \sup_{t \in \sigma(T)} |e^{\lambda t}- e^{\mu t}| < \varepsilon$$
if $|e^{\lambda}- e^{\mu}| < \delta$. The last identity is by standard functional calculus.
Therefore, the function $x \mapsto x^t$ is uniformly continuous on the positive half-line for all
$t \in \sigma(T)$.
Hence $\sigma(T) \subseteq \{0,1\}$ and so $T$ is a projection.


Consequently,
$$\varphi(1-P + e^{\lambda} P) = 1-T +e^{\lambda} T.$$
For $\lambda \rightarrow -\infty$ we get $\varphi(1-P) = 1 - T$.
Setting $P=1$ and using $\varphi(0)=0$ this implies $T=1$, and consequently
$\varphi(e^\lambda 1)= e^\lambda 1.$ In particular, $\varphi$ is $\R_+$-homogeneous.

Hence the above equality divided by $e^\lambda$ and letting $\lambda \rightarrow \infty$ yields $\varphi(P)=T$.
Thus, putting $\lambda=1$,
$$\varphi(1)= \varphi(1-P) + \varphi(P).$$
%
Now set $P=e_{11}$.
\end{proof}

We remark that in Proposition \ref{propm2}.(b) $\varphi$ is obviously actually a {\em group} homomorphism
into the image of $\varphi$. So let us also state the following variant to emphasize this fact:

\begin{corollary}  \label{corollary22}
Let $\varphi:GL(A \otimes M_2) \rightarrow GL(B)$ be an arbitrary function where
$A$ and $B$ are unital $C^*$-algebras.
Then the following are equivalent:

\begin{itemize}
\item[(a)]
$\varphi$ extends to a unital $*$-homomorphism $A \otimes M_2 \rightarrow B$.

\item[(b)]
$\varphi$ is a uniformly continuous, $*$-preserving group homomorphism satisfying (\ref{eq22}).

\end{itemize}

\end{corollary}

\begin{examples}
{\rm
\begin{itemize}

\item
The determinante $\mbox{det}:GL(M_n(\C))) \rightarrow GL(\C)$, though a
continuous $*$-preserving
group homomorphism, cannot be extended to a $*$-homomorphism because
$\det(\lambda 1) = \lambda^n$, which
is not uniformly continuous.

\item
The trivial group homomorphism $\varphi:GL(M_n(A)) \rightarrow GL(B)$, $\varphi(x)=1$,
though a uniformly continuous $*$-preserving group homomorphism, cannot be extended to a $*$-homomorphism
because $\|\varphi(1/2)\| =1$.

\end{itemize}
}
\end{examples}

\section{$KK$-theory and generators}	\label{section3}


In this section we deal with the Kasparov category $KK$. This is the category with object class being the $C^*$-algebras,
and morphism class from $C^*$-algebra $A$ to $C^*$-algebra $B$ being the Kasparov group $KK(A,B)$. Composition of morphisms
is defined to be the Kasparov product $KK(A,B) \times KK(B,C) \rightarrow KK(A,C): (f,g) \mapsto fg := f \otimes_B g$. Analogously, we have the Kasparov category $KK^G$ in the group equivariant setting with respect to a given second-countable locally compact group $G$.

By the work of J. Cuntz \cite{cuntz1984} and N. Higson \cite{higson} it became clear that Kasparov's $KK$-theory
allows a very elegant characterization 
when restricted to the class of ungraded separable $C^*$-algebras.
Cuntz noted that if $F$ is a stable, homotopy invariant, split-exact functor $F$ from the $C^*$-category $C^*$ to the abelian
groups $Ab$, then each $KK$-theory element of $KK(A,B)$ induces a map $F(A) \rightarrow F(B)$. Higson brought these findings to its final form by showing
that the Kasparov category $KK$ 
is universal in this respect in the sense that every such functor $F$ factorizes over the Kasparov category $KK$.
This fact is called the {\em universal property of $KK$-theory}.
K. Thomsen has generalized this result to the group equivariant setting, that is, to the category $KK^G$.

Quite straightforward, 
in \cite{burg_generators_KK} we described $KK^G$-theory by generators and relations based on Cuntz and Higsons's findings.
We 
denoted it by $GK$-theory (`generators $K$-theory', the group $G$ is not indicated) for better clearity.
One advantage of this basic construction is that it may be straightforwardly generalized to other
modes of equivariance, that is, to other objects than groups $G$,
for example semigroups $G$, categories $G$ and so on. Also, one may change the category $C^*$ to another category of (topological)
algebras under adaption
of the stability property, say.
Another advantage is 
that it is more elementary than Kasparov's original definition.
Its definition is also 
clearer motivated by its relative naturality,
whereas the definition of the original
$KK$-theory appears highly unmotivated at first
(without further background like the Atyiah--Singer index theory). 
Also Cuntz's picture of $KK$-theory by quasi isomorphisms in \cite{cuntzpicture} appears still rather technical and difficult.

A disadvantage of $GK$-theory is that the Kasparov product is not computed.
It remains a formal, uncomputed product $f \cdot g$.
On the other hand, this makes $GK$-theory also easy, again. 
Also, the general construction of the Kasparov product in $KK$-theory uses the indirect,
unexplicit axiom of choice.
In concrete computations the product has to be guessed, which is rather difficult.
%

We are going to briefly recall $GK$-theory. For more details see \cite{burg_generators_KK}.


\begin{definition}[$C^*$-category $C^*$]
{\rm
Let $G$ be a second-countable locally compact group, or a discrete countable inverse semigroup.
Denote by $C^*$ the category with objects being the $C^*$-algebras equipped with an action by $G$, and morphisms being the 
$G$-equivariant
$*$-homomorphisms.
}
\end{definition}

If nothing else is said, we could also allow that $G$ is another 
equivariance-inducing object like a general topological group, or a groupoid, or a category, or a semigroup and so on.

\begin{definition}[Synthetical morphisms]	\label{defsyn}
{\rm
We introduce two types of {\em synthetical morphisms}.

\begin{itemize}

\item[(a)]
For each {\em corner embedding} $c \in C^*(A,A \otimes \calk)$, that is a map
defined by $c(a)=a \otimes e$ for a one-dimensional 
projection $e \in \calk$
(where the $G$-action on $A \otimes \calk$ need not be diagonal but may be any)
introduce one synthetical morphism (inverse map, localization)
$c^{-1}:A \otimes \calk \rightarrow A$. 

\item[(b)]
For each short split exact sequence
\begin{equation}    \label{ses}
\cals: \xymatrix{
0 \ar[r] & A \ar[r]^i  & D \ar@<.5ex>[r]^f  &  B \ar[r]  \ar@<.5ex>[l]^s  & 0}
\end{equation}
in ${C^*}$
introduce one synthetical morphism $P_{\cals}^{-1}: D \rightarrow A \oplus B$ (inverse map, localization).
\end{itemize}
}
\end{definition}

\begin{definition}[Preadditive Category $W$]     \label{defW}
{\rm

Let $W$ be the preadditive category with object class $\mbox{Obj}(C^*)$.
The morphism class $W(A,B)$ from
object $A$ to object $B$ let be
the collection of all formal expressions
\begin{equation}   \label{exp1}
\pm a_{11} a_{12} \cdots a_{1 n_1} \pm \cdots \cdots \pm a_{k, 1} a_{k, 2} \cdots a_{k, n_k} ,
\end{equation}
where each {\em letter} $a_{ij}$ is either a morphism
in $C^*$ or one of the
synthetical morphisms $c^{-1}$ or $P_\cals^{-1}$ of Definition \ref{defsyn}.
Each $\pm$ stands here either for a single $+$-sign or a single $-$-sign.

We think of a {\em word} $a_{i1} \cdots a_{i, n_i}$ as a composition
of morphisms (=arrows) $a_{ij}$ going from the left to the right with start point $A$ and end point $B$, that is, as a
picture
$$\xymatrix{A= A_{i1} \ar[r]^{a_{i1}} & A_{i2}
\ar[r]^{a_{i2}} & A_{i3}  \ar[r]^{a_{i3}} &
\cdots & \ar[r]^{a_{i,n_i}}  & A_{i,n_i} = B }$$
for objects $A_{ij}$.
We require here that the range object $A_{i,j+1}$ of the morphism $a_{ij}$
coincides with the source object of the morphism $a_{i,j+1}$ for all
$ij$.
%

Composition and addition of morphisms in $W$ is given formally (i.e. freely). 
That is, we add and multiply morphisms of the from (\ref{exp1}) like in a ring
by using the distributive law.
}
\end{definition}

\begin{definition}[{$GK$-theory}]	\label{gktheory}
{\rm
The category $GK$ is defined to be additive category which
comes out when dividing 
the preadditive category $W$ 
by the following relations:
\begin{itemize}

\item[(a)]
The canonical assignment $C^* \rightarrow GK$ is a {\em functor},
i.e. we require $f g = g \circ f$ in $GK(A,C)$ for all elements $f \in C^*(A,B)$
and $g \in C^*(B,C)$.

\item[(b)]
The category $GK$ is {\em additive}, i.e. we require $p_A i_A + p_B i_B = 1_{A \oplus B}$ 
in $GK(A \oplus B, A \oplus B)$
for all natural diagrams $\xymatrix{
A \ar@<.5ex>[r]^{i_A} & A \oplus B  \ar@<.5ex>[r]^{p_B}
\ar@<.5ex>[l]^{p_A}  & B \ar@<.5ex>[l]^{i_B}  }$
(canonical injections and projections)
in $C^*$.

\item[(c)]
The category $GK$ is {\em homotopy invariant}, that is, 
every pair of 
homotopic $G$-equivariant $*$-homomorphisms $f_0,f_1:A \rightarrow B$ (homotopic within $C^*$)
satisfies the identity $f_0 = f_1$
in $GK$.

\item[(d)]
The category $GK$ is {\em stable}, that is, every corner embedding
$c$ is invertible in $GK$ with inverse $c^{-1}$ as introduced in Definition \ref{defsyn}.(a).

\item[(e)]
The category $GK$ is {\em split exact}, that is, 
for every split exact sequence (\ref{ses})
in ${C^*}$
the morphism $P_\cals := p_A i + p_B s $ 
in the following diagramm
\begin{equation}    \label{ses2}
\xymatrix{ && A \oplus B \ar@<-.5ex>[lldd]_{p_A} \ar@<.5ex>[rrdd]^{p_B}  \ar@<.5ex>[dd]^{P_\cals} &&  \\
&&&&\\
A \ar@<.5ex>[rr]^i  \ar@<-.5ex>[rruu]_{i_A}  &&  D \ar@<.5ex>[rr]^f  \ar@{.>}@<.5ex>[ll]^{t_\cals}  \ar@{.>}@<.5ex>[uu]^{{P_\cals}^{-1}}  && B  \ar@<.5ex>[ll]^s  \ar@<.5ex>[lluu]^{i_B}} 
\end{equation}
is invertible in $GK$ with inverse $P_{\cals}^{-1}$ as introduced
in Definition \ref{defsyn}.(b).
(Here, $p_A,p_B,i_A,i_B$ are the canonical projections and injections, and the dotted arrow $t_\cals$ may be ignorred here.)

\end{itemize}



}
\end{definition}

The category $GK$ is just another model for Kasparov's $KK^G$-theory:

\begin{proposition}[\cite{burg_generators_KK}]   \label{propiso}
Let $G$ be a locally compact second-countable group, or a discrete countable inverse semigroup.
Let $C^*$ be restricted to the subcategory of separable $C^*$-algebras.

Then, the categories $KK^G$ and $GK$ are isomorphic. 
\end{proposition}

\begin{proof}
Almost evident as
$KK^G$-theory %
and $GK$-theory are characterized by the same universal property. 
See \cite[Theorem 5.1]{burg_generators_KK} for more details. 
\end{proof}

In this section we are going to show that expression (\ref{exp1}) of a morphism in $GK$ may be considerably simplified.
A first simplification will be reduction of sum, where the notion word is defined in Definition \ref{defW}:

%

\begin{lemma} 
\label{lemma1}
In $GK$ we may rewrite any 
plus-signed sum $x_1 + \ldots +x_n$ of words $x_i$ as a single word
$x$.
In particular, any morphism 
in $GK$ is presentable as a difference 
$x-y$ of some words $x,y \in GK$.
\end{lemma}

\begin{proof}
By induction, it clearly suffices to show that any sum $x+y$ of two words $x,y \in GK$ 
is presentable as a single word.

Assume that we have given 
a split exact sequence
$\cals$, see (\ref{ses}), for which we consider $\vartheta:=P_\cals= p_A i + p_B s \in GK(X,Y)$
of Definition \ref{gktheory}.
Define 
$$(\vartheta \oplus \mbox{id}_X): X \oplus X \rightarrow Y \oplus X: \;$$
$$\vartheta \oplus \mbox{id}_X
:= p_A i \oplus \mbox{id}_X + p_B s \oplus 0_X
= (p_A \oplus \mbox{id}_X) (i \oplus \mbox{id}_X) + (p_B \oplus 0_X) (s \oplus 0_X).$$
Notice that $\vartheta \oplus \mbox{id}_X$ is just $P_{\calt}$
for the split exact sequence
$$\calt: \xymatrix{
0 \ar[r] & A \oplus X \ar[r]^{i \oplus \mbox{id}_X}  & D  \oplus X \ar@<.5ex>[r]^{f \oplus 0}  &  B \ar[r]  \ar@<.5ex>[l]^{s \oplus 0}  & 0}.
$$


Consider the canonical projections and embeddings
$$\xymatrix{ X \ar@<.5ex>[r]^{i_1}  &  X \oplus X  \ar@<.5ex>[l]^{p_1}  \ar@<.5ex>[r]^{p_2}  &  X , \ar@<.5ex>[l]^{i_2}
& Y &  Y \oplus X \ar[r]^{p_2'}  \ar[l]_{p_1'} & X }.$$
Set $\vartheta^{-1} := P_\cals^{-1}$. Then observe that
$$p_1 i_1 =  (\vartheta \oplus \mbox{id}_X) p_1' \vartheta^{-1} i_1, \qquad
 p_2 i_2 =  (\vartheta \oplus \mbox{id}_X) p_2' i_2,$$
%
so that with $p_1 i_1 + p_2 i_2 = \mbox{id}_{X \oplus X}$ we  get
\begin{eqnarray} 
\label{varthetasum}
(\vartheta \oplus \mbox{id}_X)^{-1} &=& p_1' \vartheta^{-1} i_1 + p_2' i_2, \\
\label{varthetasum2}
(\vartheta \oplus \mbox{id}_X) &=& ( \mbox{id}_{X \oplus X})(\vartheta \oplus \mbox{id}_X)= p_1 \vartheta i_1 + p_2 i_2'.
\end{eqnarray}

If we have given a corner embedding $\vartheta :=c \in C^*(X:=A,Y:= A \otimes \calk)$
then we set $(\vartheta \oplus \mbox{id}_X): X \oplus X \rightarrow Y \oplus X$
obvious and get again relations (\ref{varthetasum}) and (\ref{varthetasum2}).
Notice that in this case $(\vartheta \oplus \mbox{id}_X)^{-1}$ is just
the word $(\mbox{id}_{A \otimes \calk} \oplus e) d^{-1}$ 
for the corner embeddings
$d \in C^*(A \oplus X, A \otimes \calk  \oplus X \otimes \calk)$ and $e \in C^*(X, X \otimes \calk)$.


By some abuse of notation,
in the sequel we shall omit notating the primes in $p_1'$ and $p_2'$ and simply write $p_1$ and $p_2$ instead. In other words, we shall not indicate the involved
spaces $X$ and $Y$ in our notation, even when we are going to have different spaces. As already above, the index $1$ will mean projection or embedding on the first (left hand sided) coordinate,
and $2$ on the second (right hand sided) coordinate.


Let us be given two words $x_1^{\varepsilon_1} \ldots x_n^{\varepsilon_n}$ and $y_1^{\epsilon_1} \ldots y_m^{\epsilon_m}$
in $GK(X,Y)$, where $x_i \in GK(X_i,X_{i+1})$
and $y_j \in GK(Y_j,Y_{j+1})$ 
are either morphisms in $C^*$ or morphisms $P_\cals$, and let $\varepsilon_i, \epsilon_j \in \{1,-1\}$ present exponents
in case letters are invertible by synthetical inverses as defined in Definition \ref{defsyn}.
The expression $x_i^1=P_\cals^1$ is not allowed, because $P_\cals$ can be expressed
by morphisms in $C^*$.

Let $j:X \rightarrow X\oplus X$ be defined by $j(x)=(x,x)$.
Let $d:Y \oplus Y \rightarrow M_2(Y)$ be the diagonal embedding $d(x,y)= \left (\begin{matrix} x & 0 \\0 & y \end{matrix} \right )$
and $k:B \rightarrow M_2(Y)$ the corner embedding $k(x)= \left (\begin{matrix} x & 0 \\0 & 0 \end{matrix} \right )$.
Using the identities (\ref{varthetasum}) and (\ref{varthetasum2}) and their analogs, and the orthogonality relations
$i_2 p_1 = 0$ and $i_1 p_2 = 0$, the following computation shows our claim. 
Simply consider the word
\begin{eqnarray*}
 && j (x_1 \oplus \mbox{id}_X)^{\varepsilon_1} \cdots (x_n \oplus \mbox{id}_X)^{\varepsilon_n}
(\mbox{id}_X \oplus y_1)^{\epsilon_1} \cdots (\mbox{id}_X \oplus y_m)^{\epsilon_m} d k^{-1}  \\
&=&  
j  ( p_1 x^{\varepsilon_1} i_1 + p_2 i_2  ) \cdots  ( p_1 x_n^{\varepsilon_n} i_1 + p_2 i_2 )  \\
&& \cdot ( p_1 i_1 + p_2 y_1^{\epsilon_1} i_2 ) \cdots  ( p_1 i_1 + p_2 y_m^{\epsilon_m} i_2)
d k^{-1}    \\
&=&
j  ( p_1 x^{\varepsilon_1}  \cdots x_n^{\varepsilon_n} i_1 + p_2 y_1^{\epsilon_1}  \cdots  y_m^{\epsilon_m} i_2)
d k^{-1}   \\
&=&
x^{\varepsilon_1}  \cdots x_n^{\varepsilon_n} +  y_1^{\epsilon_1}  \cdots  y_m^{\epsilon_m},
\end{eqnarray*}
where for the last identity we have used that the $*$-homomorphism $i_2 d$ is homotopic to the $*$-homomorphism $i_1 d$ by rotation,
and $i_1 d k^{-1} = \mbox{id}_Y$.
%
\end{proof}

Instead of the split exactness axiom in the definition of $GK$ we may use
alternatively the following axiom without difference. 


\begin{lemma}   \label{ses3b}
Instead of introducing the synthetical arrows $P_\cals^{-1}$
in Defintion \ref{defsyn}.(b) and using axiom \ref{gktheory}.(e) we may alternatively introduce the
dotted arrow $t_\cals$ for each split exact sequence (\ref{ses}) and 
the axiomatic relations
$$i t_\cals = 1_A, \qquad t_\cals i + fs = 1_D$$
(as a replacement of Definition \ref{gktheory}.(e)) without changing the definition of $GK$.
\end{lemma}

It would not make any difference in the definition of $GK$
if we added both $P_{\cals}^{-1}$ and $t_\cals$ simultaneously,
because they automatically define each other as follows in $GK$:

\begin{lemma}   \label{ses3}
$P_\cals^{-1}$ and $t_\cals$ of diagram (\ref{ses2}) define each other as follows:
$$t_\cals = P_\cals^{-1} p_A , \quad  \quad  P_\cals^{-1} = t_\cals i_A + f i_B$$

\end{lemma}

\begin{proof}[Proof of Lemmas \ref{ses3b} and \ref{ses3}]
Let $GK$ be the category with the usual split exactness axiom involving $P_\cals$, and $GK'$ the category with the alternative split exactness
axiom involving $t_\cals$. Let $\Phi:GK \rightarrow GK'$ and $\Psi: GK' \rightarrow GK$ be the functors which are identical on $C^*$
and on the synthetical inverses of corner embeddings,
and according to the `transformation' rules defined to be
$$\Phi(P_\cals^{-1})= t_\cals i_A + f i_B, \quad \Psi(t_\cals)= P_\cals^{-1} p_A$$
for each split exact sequence $\cals$.

We remark that $s t_\cals = 0$ because $s t_\cals = s t_\cals i t_\cals = s (1- fs) t_\cals = 0$.
To see that $\Phi$ is 
well-defined we compute
$$\Phi(P_\cals) \Phi(P_\cals^{-1}) = (p_A i + p_B s)(t_\cals i_A + f i_B)= 1_{A \oplus B}, \quad \Phi(P_\cals^{-1})\Phi(P_\cals) = 1_D.$$
To show that $\Psi$ is well-defined we calculate
$$\Psi (i) \Psi(t_\cals)= i P_\cals^{-1} p_A = i_A p_A i P_\cals^{-1} p_A = i_A (P_\cals - p_B s) P_\cals^{-1} p_A
= i_A p_A = 1_A,$$
$$\Psi(t_\cals) \Psi(i) + \Psi(f) \Psi(s) = P_\cals^{-1} p_A i + fs
= P_\cals^{-1} ( P_\cals  - p_B s +  P_\cals fs) = 1_D.$$

That $\Psi$ and $\Phi$ are inverses to each other follows then from the observation
$$\Psi \circ \Phi(P_\cals^{-1}) = P_\cals^{-1} (p_A i_A + P_\cals f i_B)= P_\cals^{-1}, \quad  \Phi \circ \Psi(t_\cals) = t_\cals.$$

\end{proof}

We remark that we have also shown in the last proof that $s t_\cals =0$.
(That shows even more more clearly that $D \cong A \oplus B$ in $GK$.)
Again, the element $t_\cals$ is uniquely defined by its defining
relations. Also, Lemma \ref{lemma1} would hold if we had
introduced $t_\cals$ instead of $P_\cals^{-1}$.
All these follows immediately as a corollary from the formula
$t_\cals = P_\cals^{-1} p_A$ of Lemma \ref{ses3}.

We can always move the inverse $c^{-1}$ of a corner embedding
$c \in C^*$ to the right in a word:

\begin{lemma}  \label{lemma2}
If $f$ is a morphism in $C^*$, $c$ a corner embedding and the composition $c^{-1} f$ 
admissible,
then there exists a corner embedding $c'$ and a morphism
$f'$ in $C^*$ 
such that $c^{-1} f = f' c'^{-1}$.
(Analogously, $c^{-1} t_\cals = t_{\cals'} c'^{-1}$. Similarly, $c^{-1} P_\cals^{-1} = P_{\cals'}^{-1} \varphi^{-1} c'^{-1}$,
where $\varphi$ is the canonical isomorphism $(A \oplus B) \otimes \calk \rightarrow A \otimes \calk \oplus B \otimes \calk$.)
\end{lemma}

\begin{proof}
This follows from the commutation relation $c (f \otimes \mbox{id}_\calk) = f c'$ for the corner embeddings $c:A \rightarrow A \otimes \calk$ and $c':B \rightarrow B \otimes \calk$ and a morphism $f:A \rightarrow B$.
The case $P_\cals^{-1}$ is analog: since $\calk$ is an exact $C^*$-algebra we can tensor the diagrams (\ref{ses}) and (\ref{ses2}) with $\calk$,
then check $P_\cals c = c' \varphi P_{\cals'}$, where $c:D \rightarrow D \otimes \calk$, $c':A \oplus B \rightarrow (A \oplus B) \otimes \calk$
and $\cals' = \cals \otimes \calk$ (also with additivity, Definition \ref{gktheory}.(b)).
The case $t_\cals$ follows from 
that and $t_\cals = P_\cals^{-1} p_A$ of Lemma \ref{ses3}.
\end{proof}

A drastical simplification of morphisms in $GK$ goes by the Cuntz picture:

\begin{proposition}    \label{propmain}
Let $G$ be a locally compact second-countable group or a countable inverse semigroup and
the category $C^*$ be restricted to separable $C^*$-algebras.

Every morphism $z$ in $GK$ may be written in the form
$$z=(a d^{-1} - b) \cdot t_\cals e t_\calt \cdot 
c^{-1}$$
for some homomorphisms $a,b \in C^*$, some split exact sequences
$\cals$ and $\calt$, and some corner embeddings $c,d,e \in C^*$.

If the morphism $z$ is in $GK(A,B)$ and $B$ is unital we can omit $t_\calt$ (i.e. $t_\calt = 1$).
If $G$ is the trivial
group then $d^{-1}$ and $e$ can be omitted (i.e. $d^{-1}=e=1$).
Both simplifications can be combined simultaneously.

\end{proposition}

\begin{proof}
By the universal property of $KK^G$ and $GK$ there is an isomorphism
of categories $\hat G:KK^G \rightarrow GK$, see Proposition \ref{propiso}.
The idea is now to keep track of the formulas appearing in the proof of this fact and see how a morphism
$z\in KK^G$ is presented as $\hat G(z)$ in $GK$.
%
The original proof of the universal property of $KK$ is by Cuntz \cite{cuntz1984} and Higson
\cite{higson}, and by Thomsen \cite{thomsen} in the group equivariant setting for $KK^G$.
We shall refer here
to our exposition in the inverse semigroup equivariant setting \cite{burgUniversalKK}. 
All we shall do here may be read verbatim topological group equivariantly.




Let us be given fixed objects $A,B \in C^*$.
Assume at first that $B$ is stable, i.e. $B \cong B \otimes \calk$ in $C^*$
($\calk$ equipped with the trivial $G$-action).

In \cite[Theorem 8.5]{burgUniversalKK}, there is stated an
isomorphism
$$\Phi:\F^G(A,B) \rightarrow KK^G(A,B).$$
Here, $\F^G(A,B)$ is just the Cuntz-picture of $G$-equivariant
$KK$-theory by quasi homomorphisms and $G$-cocycles,
see \cite[Def. 7.1 and Def. 7.8]{burgUniversalKK}.
To recall it, an element $x=[\varphi_+,\varphi_-,u_+,u_-] \in
\F^G(A,B)$ is given by two $G$-equivariant $*$-homomorphisms
$\varphi_\pm:A \rightarrow \calm(B)$ and 
two $\alpha$-cocycles 
$u_\pm : G \rightarrow \calm(B)$,
see \cite[Def. 5.1]{burgUniversalKK}.

One has two split-exact sequences (for $+$ and $-$)
$$\cals: \xymatrix{ 0 \ar[r] & B \ar[r]^j & A_x \ar@<.5ex>[r]^p & A \ar[r] \ar@<.5ex>[l]^{s^{\pm}}& 0}$$
for $A_x:=\{A \oplus \calm(B)|\, \varphi_+(a) = m \mod B\}$
by \cite[Def. 9.1 and 9.4]{burgUniversalKK}.

Define the split-exact, homotopy invariant, stable functor
$F$ from $C^*$ to the abelian groups by
$$F(B)= GK(A,B) \mbox{ and } F(f:B \rightarrow C):GK(A,B) \rightarrow GK(A,C):z \mapsto z f.$$

For an $\alpha$-cocycle $u \in \calm(A)$,
recall
\cite[Def. 5.4, 6.1 and 6.2]{burgUniversalKK}
for the definition of an abelian group isomorphism
$$u_\# = F(T_{u,A})^{-1} \circ F(S_{u,A}) : F(A,\alpha) \rightarrow F(A,u \alpha u^*)$$
and corner embeddings $S_{u,A},T_{u,A} :A \rightarrow M_2(A,\delta_u)$.

As in \cite[Def. 9.4]{burgUniversalKK}, define an abelian group homomorphism
\begin{equation} \label{eq5}
\Psi_x:F(A) \rightarrow F(B): \Psi_x = {u_-}_\#^{-1} \circ F(j)^{-1}
\circ \big ( u_\# \circ F(s_+) - F(s_-) \big )
\end{equation}
(here $u$ is the cocycle for $A_x$ of \cite[Def. 9.1]{burgUniversalKK}!).

Now assume that $B$ is not necessarily stable.
In \cite[Def. 10.2]{burgUniversalKK} there appears a similar 
variant
$$\Psi'_z:F(A) \rightarrow F(B): \Psi'_z = F(c_B)^{-1}
\circ F(j_B)^{-1} \circ \Psi_{{j_B}_* {c_B}_* (\Phi^{-1}(z))}$$
of $\Psi_x$,
where $z \in KK^G(A,B)$.
Here $c_B: B \rightarrow B \otimes \calk$ is the corner embedding,
see \cite[Def. 10.1]{burgUniversalKK}, and $j_B$ appears in some
split exact sequence
$$\calt: \xymatrix{ 0 \ar[r] & B \otimes \calk \ar[r]^{j_B} & B^+ \otimes \calk \ar[r]^{p_B} & C^*(E) \otimes \calk \ar[r] & 0}$$
in \cite[Def. 10.2]{burgUniversalKK}.
The stars in ${j_B}_*$ and ${c_B}_*$ are defined in
\cite[Def. 8.6]{burgUniversalKK}.

By \cite[Def. 11.1]{burgUniversalKK} there is a natural transformation
$$\xi:KK(A,-) \rightarrow F(-): \xi_B(z) = \Psi'_z (1_{GK(A,A)}).$$

We are now applying \cite[Thm. 1.3]{burgUniversalKK} (=
\cite[Thm. 12.4]{burgUniversalKK}) to the 
canonical quotient functor $G:C^* \rightarrow GK$,
which is split-exact, homotopy
invariant and stable. The claim and proof of \cite[Thm. 12.4]{burgUniversalKK} show that there is a functor $\hat G:KK^G \rightarrow GK$ defined by
$$\hat G(z) = \xi_B(z)$$
for all $z \in KK^G(A,B)$
such that $G$ factorizes over $\hat G$ (i.e. $G= \hat G \circ G_2$
for the canonical quotient functor $G_2:C^* \rightarrow KK^G$).
This functor is an isomorphism, since $GK$ itself has the universal
properties of $KK^G$, confer \cite[5.1]{burg_generators_KK}.

In details we get
$$\hat G(z) = \xi_B(z) = \Psi'_z(1_{GK(A,A)})
= F(c_B)^{-1}
\circ F(j_B)^{-1} \circ \Psi_{{j_B}_* {c_B}_* (\Phi^{-1}(z))}
(1_{GK(A,A)}).
$$

Now observe that for the corner embedding $c_B$, the inverse
map $F(c_B)^{-1}$ is just realized by right multiplication with the
synthetical inverse $c_B^{-1}$ in $GK$.
Similarly, according to the split-exactness of $GK$ the (one-sided) inverse
map $F(j_B)^{-1}$ is just right multiplication with the synthetical
(one-sided) inverse $t_\calt$.

We choose now the $x$ from above as $x:= {j_B}_* {c_B}_* (\Phi^{-1}(z)) \in \F^G(A,B^+ \otimes \calk)$
and put formula (\ref{eq5}) into the formula of $\hat G(z)$. Here, $z$ is the given morphism in $KK^G$ that we want to present in $GK$
via $\hat G$.
Then we have
$$\hat G(z) = F(c_B)^{-1}
\circ F(j_B)^{-1} \circ
{u_-}_\#^{-1} \circ F(j)^{-1}
\circ \big ( u_\# \circ F(s_+) - F(s_-) \big ) \;
(1_{GK(A,A)})
$$
$$= 1_{GK(A,A)} \cdot (s_+ S_{u,A} T_{u,A}^{-1} - s_-) \cdot t_{\cals}
T_{u_-,A} S_{u_-,A}^{-1} \cdot t_{\calt} c_B^{-1}$$
$$= (a d^{-1} -b) \cdot t_\cals e t_\calu \cdot f^{-1} c^{-1}$$
in $GK(A,B)$ by Lemma \ref{lemma2}
for suitable homomorphisms $a,b \in C^*$, corner
embeddings $c,d,e$ and split-exact sequence $\calu$.

If $B$ is unital we can omit $j_B$ in the definition
of $\Psi'_z$.
If $G$ is trivial all cocycles satisfy $u=1$
and thus all $u_\# =1$.
\end{proof}

It is however rather difficult to bring a product of such standardized elements
as in Proposition \ref{propmain} again to such a standard form, see Cuntz \cite{cuntz1984}.
It is not really easier
than 
forming the Kasparov product of Kasparov cycles.

\begin{remark}
{\rm
A further slight simplification of the split exactness axiom could be done by observing that the split exact sequence (\ref{ses})
is isomorphic in $C^*$ to an idempotent $*$-homomorphisms $P: D \rightarrow D$ (translation is $P=fs$).
Then split exactness just says that every idempotent $P \in C^*$ has an orthogonal split 
$t_\cals:D \rightarrow \mbox{ker}(P)$ in $GK$ (orthogonal projection: $t_\cals i = 1_D-P$).
}
\end{remark}


\section{$E$-continuity and Hausdorff property}		\label{section4}

In this section we shall see that the groupoid associated
to an inverse semigroup is Hausdorff if and only if the inverse
semigroup is $E$-continuous.
This condition is technically easier and more intrinsic to the inverse semigroup. We shall see that $E$-continuity is a necessary and
sufficient condition to define a non-degenerate $C_0(X)$-compatible $C_0(X)$-valued $L^2(G)$-module.  
%
%



Let $G$ be a discrete inverse semigroup.

\begin{definition}[$E$ and $X$]
{\rm
Let $E$ denote the subset of idempotent elements of $G$.
The free universal {\em abelian} $C^*$-algebra $C^*(E)$ generated by
the commuting self-adjoint projections of $E$ has a totally disconnected Gelfand spectrum $X$.
That is we have $C^*(E) \cong C_0(X)$.
Under this isomorphism we identify $E$ as a subset of $C_0(X)$
(under the formula $e(x)= x(e)$). 
To this end, we also use the suggestive notation $1_e \in C_0(X)$ 
for the corresponding element of $e \in E$ in $C_0(X)$.
We write ``$x \in e$" for $x \in X$ and $e \in E$ iff $x$
is an element of the support of $1_e \in C_0(X)$ (also denoted by $\mbox{carrier}(1_e)$).
For $e,f \in E$ we use the usual order $e \le f$ in a $C^*$-algebra.
This {\em order} can be extended to $G$ by saying that $g \le h$
for $g,h \in G$
iff $g = h g^* g$
(or equivalently iff $g = g g^* h$).
}
\end{definition}

\begin{definition}[$G$-action]	\label{defgalgebra}
{\rm
In this note we understand under a {\em $G$-action} on a $C^*$-algebra $A$ a semigroup
homomorphism $\alpha: G \rightarrow \mbox{End(A)}$
such that $\alpha_e(a) b = a \alpha_e(b)$ ({compatibility}) for all $e$ in $E$.
In this case, $A$ is called a {\em $G$-algebra}.
A $G$-action on a Hilbert $A$-module $\cale$ is a semigroup
homomorphism $U:G \rightarrow \mbox{LinMaps}(\cale)$ (linear maps)
such that $U_e$ is an adjoint-able operator for all $e \in E$,
and
$$\langle U_g(\xi),U_g(\eta)\rangle = g(\langle \xi,\eta\rangle), \;\; U_g(\xi a) = U_g(\xi) \alpha_g(a), \;\; U_e(\xi) a = \xi \alpha_e(a)$$
(the last identity being called 
{\em compatibility} or {\em $C_0(X)$-compatibility} of $U$)
for all $\xi,\eta \in \cale, a \in A, g \in G$ and $e \in E$.
Then $\cale$ is called a (compatible) {\em $G$-Hilbert $A$-module}.
Often we write the $G$-action in the form $g(\xi):=U_g(\xi)$
and $g(a):= \alpha_g(a)$.
}
\end{definition}

\begin{definition}[$G$-action on $X$]  
{\rm
The $C^*$-algebra $C_0(X)$ is equipped
with the $G$-action $g(1_e):= 1_{g e g^*}$
for $e \in E, g \in G$.
This $G$-action may be extended to the bigger
$C^*$-algebra $\ell^\infty(X)$ by setting
$(g(f))(x):= 1_{\{x \cdot g \neq 0
\}} f(x \cdot g)$ for $g \in G, f \in \ell^\infty(X)$ and characters $x \in X$,
where the (possibly zero) character $x \cdot g:C^*(E) \rightarrow \C$ is defined
by $(x \cdot g)(e)= x(g e g^*)$ for all $e \in E$.
}
\end{definition}

We are going to recall the $E$-continuity property of an inverse semigroup.
For more details see \cite{burgAttempts}.
In the next few paragraphs (until Lemma \ref{lemmalinindepen}) 
we shall identify elements $e \in E$ with
their corresponding characteristic functions
$1_e$ in $C_0(X)$.
Write $\mbox{Alg}^*(E)$ for the dense $*$-subalgebra of $C_0(X)$ generated by the characteristic functions
$1_e$ for all $e \in E$.
Moreover,
write $\bigvee_i f_i :X \rightarrow \C$ for the pointwise supremum of a family of functions $f_i: X \rightarrow \C$.

\begin{definition}   \label{definitionEcontinuous}
{\rm
An inverse semigroup $G$ is called {\em $E$-continuous} if
the function $\bigvee \{ e \in E |\, e \le g\} \in \C^X$ (in precise notation: $\bigvee \{1_e \in C_0(X) |\,e\in E,\, e \le g\} \in \C^X$) 
is a {\em continuous} function in $C_0(X)$
for all $g \in G$.
}
\end{definition}

A simple compactness argument shows the following, see \cite{burgAttempts}:

\begin{lemma}   \label{lemmaEcontinuity}
An inverse semigroup $G$ is {$E$-continuous} if and only if for every $g \in G$ there exists a finite
subset $F \subseteq E$ such that
$\bigvee \{e \in E |\, e \le g\} =\bigvee \{e \in F |\, e \le g\}$.
\end{lemma}



\begin{definition}[Compatible $C_0(X)$-valued $L^2(G)$-module]     \label{defCompatibleL2}
{\rm
Let $G$ be an $E$-continuous inverse semigroup.
Write $c$ for the linear span of all functions $\varphi_g: G \rightarrow \C$ (in the linear space $\C^G$) defined by
\begin{eqnarray*}
\varphi_g(t) & :=& 1_{\{t \le g\}} 
\end{eqnarray*}
(characteristic function)
for all $g,t \in G$.
Endow $c$ with the $G$-action $g(\varphi_h) := \varphi_{g h}$ for all $g,h \in G$.
Turn $c$ to an $\mbox{Alg}^*(E)$-module by setting $\xi e := e(\xi)$ for all $\xi \in c$ and $e \in E$.
Define an $\mbox{Alg}^*(E)$-valued inner product on $c$ by
\begin{eqnarray}  \label{identinnerprodphi}
\langle \varphi_g, \varphi_h \rangle &:=& \bigvee \{e \in E \,|\, eg = e h ,
\,e \le g g^* h h^* \}.
\end{eqnarray}
The norm completion of $c$ is a 
$G$-Hilbert $C_0(X)$-module denoted by $\widehat{\ell^2}(G)$.
}
\end{definition}

\begin{lemma}[\cite{burgAttempts}]    \label{lemmalinindepen}
The vectors $(\varphi_g)_{g \in G} \subseteq \widehat{\ell^2}(G)$ are linearly independent.
\end{lemma}

We recall the well-known topological
groupoid associated to an inverse semigroup by Paterson \cite{paterson}:

\begin{definition}[{Groupoid associated to an inverse semigroup}]
{\rm
Let $G$ be a discrete inverse semigroup and $X$ the Gelfand spectrum of $C^*(E)$.
Consider the topological subspace $G * X
= \{(g,x) \in G \times X|\, g \in G, \,x \in g^* g\}$ of the topological space $G \times X$ (product topology with $G$ having the discrete topology). Two points $(g,x),(h,y)$ in $G * X$
are called {\em equivalent}, also denoted $(g,x) \equiv (h,y)$,
iff $x=y$ and $g e = h e$ for some $e \in E$
with $x \in e$.
Let $\pi: G * X \rightarrow G * X /\equiv$ denote the set-theoretical
quotient map. The quotient is a groupoid under the
multiplication: $\pi(g,x) \pi(h,y)
= \pi(g h, y)$ if and only if 
for all $e \in E$
such that $y \in e$ one has $x \in (h e)(h e)^*$. Otherwise the composition is declared
to be undefined.

We now regard the
quotient $G * X /\equiv$ as a topological groupoid under the quotient topology and call it the
{\em groupoid asscociated to the inverse semigroup $G$}.
(
Recall that a subset $Y \subseteq G * X /\equiv$ is declared
to be open if and only if $\pi^{-1}(Y)$ is open.)
}
\end{definition}

Usually the groupoid associated to $G$ is a non-Hausdorff topological space.
We are going to prove that the Hausdorff condition is equivalent to $E$-continuity of $G$.



\begin{lemma}
The sets of the form $\pi(g \times U)$, where $g \in G$
and $U \subseteq X$ is an open subset of $X$ with $U \subseteq \mbox{carrier}(g^*g)$, are open and generate the topology of
$G * X/\equiv$. (Here $g \times U:= \{g \} \times U$.)

\end{lemma}

\begin{proof}
We claim that the inverse $\pi^{-1}( \pi(g \times U))$ is 
open.
Indeed if $(h,x) \in \pi^{-1}( \pi(g \times U))$
then it is equivalent to some $(g,x) \in g \times U$.
Hence there exists some $e \in E$ with $x \in e$ and $h e = g e$. Let $V= \mbox{carrier}(e) \cap U \cap \mbox{carrier}(h^*h)$.
Then $h \times V$ is an open subset of $\pi^{-1}( \pi(g \times U))$
containing $(h,x)$.


If $\pi^{-1}(O)$ is open and contains the point $(g,x)$
together with its open neighborhood $g \times U$
then $\pi^{-1} (\pi(g \times U)) \subseteq \pi^{-1}(O)$.
Thus $\pi(g \times U) \subseteq O$. Hence such sets generate
the topology.
\end{proof}

We call $\pi(g \times U)$ the {\em open set
in $G * X/\equiv$ generated
by $g \times U$}.

\begin{lemma}
If 
$G$ is $E$-continuous then its associated groupoid is Hausdorff.
\end{lemma}

\begin{proof}
Let $(g,x),(h,x) \in G * X$ be two points such that $(g,x) \not \equiv (h,x)$.
Then for all $e \in E$ with $x \in e$ and $e \le g^*g h^*h$ one has $g e \neq h e$,
and so $e \not \le h^* g$.
Since $G$ is $E$-continuous the function
$F:= \bigvee_{f \in E, f \le h^*g} f$ is continuous.
Note that $x \notin F$.

Let $t \subseteq X$ be the (open!) complement of the carrier of $F$.
Consider $U_g:=\{g\} \times t \cap \mbox{carrier}(g^*g)$ and $U_h:=\{h\} \times t \cap \mbox{carrier}(h^*h)$.
Clearly $x \in t$ and so $(g,x) \in U_g$ and $(h,x) \in U_h$.


Consider the open subsets $W_g$ and $W_h$ that $U_g$ and $U_h$ generate in $G * X/\equiv$.
Assume $W_g$ and $W_h$ would
intersect. Then there are $(g,y) \in U_g$, $(h,z) \in U_h$ such that $(g,y)\equiv (h,z)$.
That is, there is a $e \in E$ such that $y=z \in e$, $e \le g^*g h^*h$ and $g e = h e$.
Hence $y \in e \le F$.
By definition of $U_g$ one has also certainly
$y \in t$.
A contradiction.
%
This shows that $W_g$ and $W_h$ are disjoint neighborhoods
which separate $(g,x)$ and $(h,x)$.
%
%
\end{proof}

\begin{lemma}
If its associated groupoid is Hausdorff
then $G$ is $E$-continuous.
\end{lemma}

\begin{proof}



Let $g \in G$.
Assume the projection $F:=\bigvee_{f \in E, \,f \le g} f$ would be discontinuous,
say in the point $x \in X$.

Then for any neighborhoods $U$ of $x$ there is at least one $f \le g$
($f \in E$)
such that $U$ has nonempty intersection with the carrier of $f$. 
On the other hand 
$x$ is not in the carrier of any $f \in E$ with $f \le g$,
because there
$F$ is continuous.


Consider the points $(g,x)$ and $(g^*g,x)$ in $G * X$.
They must be distinct in the quotient $G * X/\equiv$
because assuming to the contrary the existence of some $e \in E$
with $x \in e$ and $g^*g e = g e$ would imply $g^*g e \le g$;
a contradiction to what we said above.


Let $U \subseteq \mbox{carrier}(g^*g) \subseteq X$ be an open neighborhood of $x$.
Consider the open neighborhoods $W_g$ and $W_{g^*g}$ 
in $G * X / \equiv$ generated by $\{g\} \times U$ and $\{g^*g\} \times U$.
As remarked above we may
choose $y \in U, f \in E$ such that $y \in f$ 
and $f \le g$.
Then $(g,y)$ and $(g^*g,y)$ are equivalent because $y \in f$ and $g^*g f = g f$.

Hence $W_g$, $W_{g^*g}$ intersect. Hence $(g,x)$ and $(g^*g,x)$ cannot be separated.
Contradiction.
\end{proof}

\begin{corollary}
An inverse semigroup 
is $E$-continuous
if and only if its associated groupoid is Hausdorff.
\end{corollary}

We have seen in Definition \ref{defCompatibleL2} that
for $E$-continuous inverse semigroups there exist
non-degenerate compatible $L^2(G)$-modules
with coefficients in $C_0(X)$.
The next example 
indicates
that we cannot construct
such $L^2(G)$-modules for $E$-discontinuous
inverse semigroups.

Before that,
for the discussion of another $L^2(G)$-module, 
we recall the following discretized coefficient algebra $\varepsilon(E)$ of $C_0(X)$.

\begin{definition}[Discretized coefficient algebra $\varepsilon(E)$ of $C_0(X)$]
{\rm
Recall that there exists a map $\epsilon:E \rightarrow X$
assigning to each $e \in E$ the character $\epsilon_e$ on $C^*(E)$
determined by the formula $\epsilon_e(f) = 1_{\{f \ge e\}}$ for every $f \in E$.
The image $\epsilon(E)$ is dense in $X$, see \cite{paterson} or \cite[3.2]{khoshkamskandalisRegular}.
%
We have a $G$-invariant sub-$C^*$-algebra $$\varepsilon(E) := c_0 \big (\epsilon(E) \big ) \subseteq \ell^\infty(X)$$
(complex-valued functions on the image of $\epsilon$ vanishing at infinity).
Given $e \in E$, we write $\varepsilon_e$
for the characteristic one-point supported
function $1_{\{\epsilon_e\}} \in \varepsilon(E) \subseteq \ell^\infty(X)$.  
One 
checks that $G$ acts through $g(\varepsilon_e)= \varepsilon_{g e g^*}$
if $e \le g^* g$, and $g(\varepsilon_e)=0$ otherwise.
}
\end{definition}

%

\begin{example}[{Elementary abelian $E$-discontinuous example}]   \label{example1}
{\rm
%
Let us discuss one of the most simplest examples of an (even abelian) inverse
semigroup $G$ which is not $E$-continuous.
Let $G=\{1,S,e_1,e_2,e_3,\ldots\}$
consist of an identity element $1$, a strictly increasing sequence of projections $e_1 < e_2 < e_3 < \ldots < 1$, and a symmetry
$S \neq 1$ (i.e. $S^2=1, S^*=S$) 
such that $S e_n = e_n S = e_n$ for all $n \ge 1$.
(A concrete representation of $G$ on a direct sum Hilbert space $H \oplus H$ may be given as
$1=\mbox{id}_H \oplus \mbox{id}_H$, $S = s \oplus \mbox{id}_{H}$ with $s$ a symmetry and $e_n \le 0 \oplus \mbox{id}_{H}$.) 

The associated $C^*$-algebra $C^*(G)$ is an AF-algebra. Indeed it is 
the union of its finite-dimensional sub-$C^*$-algebras $A_n$ generated by $\{1,S,e_1,\ldots, e_n\}$.
One has $A_n \cong \C^{n+2}$ for all $n \ge 0$.
The two generating projections of $A_0 \cong \C^2$
are $(1 \pm S)/2$. The projection
$(1-S)/2$ is orthogonal to all projections $e_n$, and $e_n < (1+S)/2$.
Hence
$K_0(C^*(G))= \bigoplus_\N \Z \sqcup \{1\}$
(here $1$ denotes an adjoint unit).

If we compare this with $\varepsilon(E) \rtimes G$ then we have that it is the union of the sub-$C^*$-algebras $B_n$ generated
by $\varepsilon_{1} \rtimes 1, \varepsilon_{1} \rtimes S,
\varepsilon_{e_1} \rtimes e_1, \ldots , \varepsilon_{e_n} \rtimes e_n$.
Again $B_n \cong \C^{n+2}$.
But
$K_0(\varepsilon(E) \rtimes G) =\bigoplus_\N \Z$
as the projection
$\varepsilon_1 \rtimes (1+S)/2$ is orthogonal
to all projections $\varepsilon_{e_n} \rtimes e_n$.

We are now coming to the most important point, namely that it appears not possible to construct a non-degenerate $C_0(X)$-compatible $C_0(X)$-valued $L^2(G)$-module.
Somehow we should have some sort of generators $\delta_1,\delta_S,\delta_{e_1}
,\ldots,\delta_{e_n},\ldots$
of the module.
The $G$-action should be $g(\delta_h)
= \delta_{g h}$ to be regarded as an $L^2(G)$-module. 
By compatibility of the module product we naturally have $\delta_g e = \delta_g \cdot e(1)=  e(\delta_g) \cdot 1 = \delta_{e g}$ 
for $e \in E \subseteq C_0(X)$.
Naturally we should choose $\langle \delta_{e_n}, \delta_{e_n} \rangle =e_n$
for the inner product. By compatibility
of the inner product
we have $\langle \delta_S , \delta_S\rangle e_n =
\langle \delta_S e_n , \delta_S e_n \rangle = e_n$
for all $n\ge 1$.
Consequently $C_0(X) \ni \langle \delta_S , \delta_S\rangle = 1$ 
(because the carriers of the elements $e \in E$ generate the topology of $X$)
and similarly $\langle \delta_S , \delta_1\rangle = \langle \delta_1 , \delta_1\rangle = 1$.
But then $\| \delta_1-\delta_S\| = 0$ and the module
degenerates.

Let us discuss another module.
We may construct the non-degenerate $\varepsilon(E)$-valued $L^2(G)$-module of \cite[Def. 5.5]{burgNoteBC}.
The generators are the characteristic functions $\delta_g:G \rightarrow \C$
with $\delta_g(h)=1_{\{g=h\}}$ for $g,h \in G$. The $G$-action is given by $h(\delta_g)= 1_{\{h^* h \ge g g^*\}} \delta_{h g}$. 
The inner product is determined by $\langle \delta_g,\delta_h\rangle  = 1_{\{g=h\}}$, $\langle \delta_p, \delta_p \rangle = \varepsilon_p$ for the projections $p$ in $G$, and $\langle \delta_S , \delta_S \rangle = \varepsilon_1$.
The module product computes as $\delta_p \varepsilon_q = 1_{\{p=q\}}$ for projections $p,q$, and $\delta_S \varepsilon_1 = \delta_S$ and $\delta_S \varepsilon_{e_n} = 0$.
}
\end{example}

\begin{example}[Dense $E$-discontinuity example]
{\rm
In Example \ref{example1} we had some kind of $E$-discontinuity only at $S$ (or we may say at $1$). We may construct
such an $E$-discontinuity at every $e$ in $E$ by the same method.
Start with a given inverse semigroup $G=E$ consisting only of projections.
Adjoin to $G$ for every $e$ in $E$ a symmetry
$S_e$ such that $S_e f = f S_e = f$ for all $f < e$.
Other relations we do not add.
The resulting inverse semigroup $G$
is $E$-discontinuous in those $S_e$
in the sense that $\bigvee_{f \in E, f \le S_e}
f$ is discontinuous where $e$ has no precursor
$f < e$. If no element of $E$ has a precursor
in $E$ then the $E$-discontinuity points
are dense in $X$ (at the points 
$\epsilon_e$ we may say, which form a dense subset of $X$).
}
\end{example}

\begin{example}[{Finitely presented $E$-discontinuous inverse semigroup}]
{\rm
A finitely presented $E$-discontinuous
inverse semigroup may be defined
as follows.
%
Consider the finitely presented inverse semigroup
$$G=\langle t,l,e \,| \,tl = lt,\,t^* l = l t^*,\, t e =e, \, t^* e = e \rangle.$$
That is $t$ and $t^*$ commute with $l$ and $l^*$, and $e$ absorbs $t$ and $t^*$.

Between $l$ and $e$ we have no relations,
they are free in $G$,
so that we get
infinitely many distinct projections
$$p_0:=e, \; p_1:= lee^* l^*, \; p_2:=llee^* l^* l^*,
\ldots,\; p_n := l^{n} e e^* {l^*}^{n},\ldots$$
in $G$.
Now
$t p_n = p_n$ by the defining relations of $G$.
The projections $p_n$ cannot be compared among each other, i.e. $p_n \le p_m$ implies $n=m$.
Hence the criterion for $E$-continuity of
Lemma \ref{lemmaEcontinuity} fails for $t$,
as the supremum of $\{e \in E|\, e\le t\}$ 
will
not be attained at a finite set of projections
of $E$. 
To see this, let us first note that
we have no single projection $q \in E$
such that $t \ge q \ge p_0 , p_1, p_2, \ldots$.
Indeed, every such projection $q$ would require to include the letter $e$ to obtain $t \ge q$,
and consequently any letter $t$ or $t^*$ in $q$ would be absorbed by $e$. So $q$ 
would allow a presentation with letters $l$ and $e$ and their adjoints only, and such a $q \ge  p_1,p_2 ,p_3, \ldots$
as required does not exist.
One can similarly argue that we also cannot choose $q_1,\ldots,q_n \in E$ such that
$t \ge q_1 \vee \ldots \vee q_n \ge p_0,p_1,p_2,\ldots$.
So by Lemma \ref{lemmaEcontinuity} we 
get that $G$ is {\em not} $E$-continuous.
%


%
%
}
\end{example}

\begin{remark}[Baum--Connes map for inverse semigroups] 
{\rm
In \cite{burgAttempts} we have tried to define a Baum--Connes map for inverse semigroup crossed products parallel to the method
of Meyer and Nest in \cite{meyernest2006} for group crossed products, which automatically would include some theoretical
method to compute the left hand side of the Baum--Connes map. On that way, $C_0(X)$-compatible Hilbert modules and their $KK$-theory
appeared the better choice than the corresponding, $C_0(X)$-structure ignorring incompatible tools.
Thus $C_0(X)$ is the natural coefficient algebra. But since $L^2(G)$-spaces are in the center and the core
of any 
Baum--Connes theory,
and Example \ref{example1} shows that a compatible $C_0(X)$-valued $L^2(G)$-module
requires $E$-continuity of $G$, it appears not possible to overcome the $E$-discontinuity barrier when defining a 
Baum--Connes map, at least not with the known (group) $L^2(G)$-space methods. 
That is, as soon as the associated groupoid of $G$ is non-Hausdorff the method fails.
More generally, Tu \cite{tunonhausdorffgroupoid} has tried to develop a Baum--Connes theory for non-Hausdorff groupoids,
and came to the same conclusion that for non-Hausdorff groupoids the known methods fail,
even one may be able to formally write down the Baum--Connes map also for non-Hausdorff groupoids.
}
\end{remark}

\begin{remark}[Baum--Connes theory for discreticized crossed products]
{\rm
Whereas we have no approach to handle the $K$-theory of a crossed product $A \rtimes G$ for an inverse semigroup $G$,
we have a Baum--Connes map and additionally at least theoretically an approach to treat the $K$-theory of $(\varepsilon(E) \otimes_{C_0(X)} A)\rtimes G$ by \cite{burgNoteBC}.
Even though 
the $K$-theories of the latter two crossed products are obviously different in general, they might have some aspects in common 
in certain good interesting cases
as the latter two
crossed products are also similar.
For example, if $G=E$ consists only of projections then both $C_0(X) \rtimes E$ and $\varepsilon(E) \rtimes E$ are 
the direct limit of canonically $*$-isomorphic finite dimensional sub-$C^*$-algebras. 
Only the direct limit
embedding maps are different in both cases. 
Hence their $K$-theories are still similar. For example, in both cases the
$K_0$-groups
are {\em infinitely generated} and the $K_1$ groups are zero if $E$ is infinite.
That is, being infinitely generated is a common quality of the $K$-theory of both crossed products.
}
\end{remark}

\begin{example}
{\rm
That being said, let us remark that the discretized crossed product and
the usual crossed product may however also be rather distinct.
Write for example the Cuntz algebra $\calo_n$ as the inverse semigroup crossed
product $\calo_n \cong A \rtimes G$ (Sieben's crossed product, which is the universal crossed product subject to the relations $e(a) \rtimes g \equiv a \rtimes e g$ for all $a \in A, e \in E, g \in G$), where
$G$ is defined to be the inverse semigroup $G \subseteq \calo_n$ generated by the standard generators
$S_1, \ldots, S_n$ of the Cuntz algebra, and $A \subseteq \calo_n$ denotes the
smallest $G$-invariant $C^*$-subalgebra of the Cuntz algebra generated by the identity
$1 \in \calo_n$ under the (incompatible) $G$-action $g(a) = g a g^*$ for $a \in \calo_n, g \in G$.
Note that $A$ is the commutative $G$-algebra (in the sense of Definition \ref{defgalgebra})
generated by the elements of the form $g g^*$ for $g \in G$.
The isomorphism is $\varphi:\calo_n \rightarrow A \rtimes G: \varphi(S_i)= 1 \rtimes S_i$.
Then we have that 
$$0 = (\varepsilon(E) \otimes_{C_0(X)} A)\rtimes G  \neq A \rtimes G = \calo_n,$$
because in the left hand sided crossed product we have
$$(\varepsilon_1 \otimes 1) \rtimes 1 = \big(\varepsilon_1 \otimes (S_1 S_1^* + \ldots + S_n S_n^*) \big ) \rtimes 1 = 0$$
as $S_i S_i^*(\varepsilon_1) = 0$ (action of $S_i S_i^*$ on $\varepsilon_1$) for all $i$, and by similar reasoning
$(\varepsilon_e \otimes 1) \rtimes 1 = 0$ for all $e \in E$.
We see thus that the discretized crossed prodcut is not an approximation
of the crossed product $A \rtimes G$ at all
as it collapses to zero. (As already the discretized coefficient algebra $\varepsilon(E) \otimes_{C_0(X)} A$ is zero). Still the $K$-theory of both crossed products
is finitely generated. But this need not be in general true, as we may
replace $A$ by an infinite sum of copies of $A$, and so $K_0((\bigoplus_\N A ) \rtimes
G) = \bigoplus_\N K_0(A  \rtimes
G)$ is infinitely generated whereas the $K$-theory of the discretized crossed product
is an infinite sum of zeros, so zero and thus finitely generated. 
}
\end{example}

\bibliographystyle{plain}
\bibliography{references}

\end{document}